\documentclass{rmmcart}
\usepackage{amsmath,amsthm,amssymb}
\usepackage[utf8]{inputenc}
\usepackage{color}
\usepackage{multicol}
\usepackage{caption}
\usepackage{float}
\usepackage{bm}
\restylefloat{table}

\title{Augmented Generalized Happy Functions} 
\author{B. Baker Swart}
\address{Department of Mathematics \& Computer Science, The Citadel, 29409}
\email{breeanne.swart@citadel.edu}
\author{K. A. Beck}
\address{Department of Mathematics \& Computer Science, Saint Mary's College of California, 94575}
\email{kab24@stmarys-ca.edu}
\author{S. Crook}
\address{Division of Mathematics, Engineering, \& Computer Science, Loras College, 52001}
\email{susan.crook@loras.edu}
\author{C. Eubanks-Turner}
\address{Department of Mathematics, Loyola Marymount University, 90045}
\email{ceturner@lmu.edu}
\author{H. G. Grundman}
\address{Department of Mathematics,
Bryn Mawr College, 19010}
\email{grundman@brynmawr.edu}
\author{M. Mei}
\address{Department of Mathematics \& Computer Science, Denison University, 43023}
\email{meim@denison.edu}
\author{L. Zack} 
\address{Department of Mathematics \& Computer Science, High Point University, 27262}
\email{lzack@highpoint.edu}
\thanks{Initial work on the project was supported by the National Science Foundation grant \#DMS 1239280.}
\newcommand{\ZZ}{\mathbb{Z}}
\newcommand{\Znn}{\mathbb{Z}_{\geq 0}}

\newcommand{\Scb}{\ensuremath{S_{[c,b]}}}

\newcommand{\Ucb}{\ensuremath{U_{[c,b]}}}

\newtheorem{theorem}{Theorem}[section]
\newtheorem{conjecture}[theorem]{Conjecture}

\newtheorem{lemma}[theorem]{Lemma}

\theoremstyle{remark}
\newtheorem{definition}[theorem]{Definition}

\keywords{happy numbers, iteration, integer functions}
\subjclass{11A63}

\begin{document}
\begin{abstract}
An augmented generalized happy function, $\Scb$ maps a positive integer to the sum of the squares of its base-$b$ digits and a non-negative integer $c$.  A positive integer $u$ is in a 
{\em cycle} of $\Scb$ if, for some positive integer $k$, $\Scb^k(u) = u$ and for positive integers $v$ and $w$, $v$ is {\em $w$-attracted} for $\Scb$ if, for some non-negative integer $\ell$, $\Scb^\ell(v) = w$. In this paper, we prove that for each $c\geq 0$ and $b \geq 2$, and for any $u$ in a cycle of $\Scb$, (1) if $b$ is even, then there exist arbitrarily long sequences of consecutive $u$-attracted integers and (2) if $b$ is odd, then there exist arbitrarily long sequences of 2-consecutive $u$-attracted integ\end{abstract}

\maketitle

%%%%%S:Introduction%%%%%
%%%%%S:Introduction%%%%%
\section{Introduction}\label{S:Introduction}
%%%%%S:Introduction%%%%%
%%%%%S:Introduction%%%%%

Letting $S_2$ be the function that takes a positive integer to the sum of the squares of its (base 10) digits, a positive integer $a$ is said to be a \emph{happy number} if $S_2^k(a) = 1$ for some $k\in \ZZ^+$~\cite{guy,honsberger}.
These ideas were generalized in~\cite{genhappy} as follows:
Fix an integer $b \geq 2$, and let $a = \sum_{i=0}^n a_i b^i$, where $0 \leq a_i \leq b-1$ are integers. For each integer $e \geq 2$, define the function $S_{e,b}: {\ZZ}^+ \rightarrow {\ZZ}^+$ by
\[S_{e,b}(a) = S_{e,b}\left(\sum_{i=0}^n a_i b^i\right) = \sum_{i=0}^n a_i^e.\]
If $S_{e,b}^k(a)=1$ for some $k\in\ZZ^+$, then $a$ is called an \emph{$e$-power $b$-happy number}.

We further generalize these functions by allowing the addition of a constant after taking the sum of the powers of the digits. 
(Throughout this work, all parameters are assumed to be integers.)
\begin{definition}\label{D:gen}
Fix integers $c\geq 0$ and $b \geq 2$. Let $a = \sum_{i=0}^n a_i b^i$, where $0 \leq a_i \leq b-1$ are integers. For each integer $e \geq 2$, define the {\it augmented generalized happy function} $S_{e,b,c}: {\ZZ}^+ \rightarrow {\ZZ}^+$, by
\[S_{e,b,c}(a) = c + S_{e,b}(a) = c + \sum_{i=0}^n a_i^e.\]
\end{definition}

In Section~\ref{S:Properties}, we examine various properties of the function $S_{2,b,c}$, which,  
for ease of notation, we denote by $\Scb$.  
In Section~\ref{S:cbGoodness}, we state and prove Theorem~\ref{T:happysequence}, an analogue to the existence of arbitrarily long sequences of consecutive happy numbers.
Although this result is quite general, it leaves a particular case unresolved, which we make explicit in Conjecture~\ref{C:happysequencebetter}, and in Section~\ref{S:conj},
we prove that the conjecture holds for small values of $c$ and $b$.

\section{Properties of $\boldsymbol\Scb$}
\label{S:Properties}

In this section, we consider the function $\Scb = S_{2,b,c}$. 
Note that $S_{[0,10]} = S_2$, as defined in Section~\ref{S:Introduction}. 
\begin{definition}
Fix $c\geq 0$, $b \geq 2$, and $a \geq 1$. We say that $a$ is a {\em fixed point} of
$\Scb$ if $\Scb(a) = a$, and that $a$ is in a {\em cycle} of $\Scb$ if $\Scb^k(a) = a$ for some $k\in \ZZ^+$.  The smallest such $k$ is called the {\em length} of the cycle.
\end{definition}

As is well-known (see, for example,~\cite{honsberger}), $S_2$ has exactly one fixed point and one nontrivial cycle.  The standard proof of this uses a lemma similar to the following, which is a generalization for the function $\Scb$.

\begin{lemma}\label{L:bDescent}
Given $c \geq 0$ and $b \geq 2$, there exists a constant $m$ such that for each $a\geq b^m$,
$\Scb(a) < a$.  In particular, this inequality holds for any $m \in \ZZ^+$ such that
$b^m > b^2 - 3b + 3 + c$.  
\end{lemma}

\begin{proof}
Let $m\in \ZZ^+$ such that $b^m > b^2 - 3b + 3 + c$ and let $a \geq b^m$.
Then $a = \sum_{i=0}^m a_ib^i$ for some $0\leq a_i \leq b - 1$ with $a_m \neq 0$.
Thus,

\begin{align*}
a-\Scb(a) &= \sum_{i=0}^n a_i b^i - \left(c + \sum_{i=0}^n a_i^2\right) \\
%  &= \sum_{i=0}^n a_i (b^i-a_i)-c \\
 &= a_n\left(b^n-a_n\right) + \sum_{i=1}^{n-1} a_i \left(b^i-a_i\right) + a_0(1-a_0)-c\\
 & \geq 1\left(b^m-1\right) + 0 + (b-1)(1 - (b-1))-c\\
 & = b^m - b^2 + 3b - 3 - c \\
 & > 0.
\end{align*}
Therefore, $\Scb(a) < a$ for all $a\geq b^m$.
\end{proof}

\begin{table}[tbh]
\begin{center}
{\small
\begin{tabular}{|c|l|}\hline
$c$&Fixed Points and Cycles of $S_{[c,10]}$ \\
\hline\hline
%========
0& $1 \to 1$\\
& $4\to 16 \to 37 \to 58 \to 89 \to 145 \to 42 \to 20 \to 4 $\\
\hline
%========
1& $6\to 37 \to 59 \to 107 \to 51 \to 27 \to 54 \to 42 \to 21 \to 6$\\
& $35 \to 35$\\
& $75 \to 75$ \\
\hline
%========
2& $28 \to 70 \to 51 \to 28$ \\
& $29 \to 87 \to 115 \to 29$ \\
\hline
%========
3& $7 \to 52 \to 32 \to 16 \to 40 \to 19 \to 85 \to 92 \to 88 \to 131 \to 14 \to 20 \to 7$ \\
& $13 \to 13$\\
& $93 \to 93$\\
\hline
%========
4& $6 \to 40 \to 20 \to 8 \to 68 \to 104 \to 21 \to 9 \to 85 \to 93 \to 94 \to 101 \to 6$ \\
& $24 \to 24$\\
& $45 \to 45$\\
& $65 \to 65$\\
& $84 \to 84$\\
\hline
%========
5& $ 15 \to 31 \to 15$\\
& $55 \to 55$\\
\hline
%========
6
& $16 \to 43 \to 31 \to 16$ \\
& $19 \to 88 \to 134 \to 32 \to 19$ \\
\hline
%========
7 & $9 \to 88 \to 135 \to 42 \to 27 \to 60 \to 43 \to 32 \to 20 \to 11 \to 9$\\
& $12 \to 12$\\
& $36 \to 52 \to 36$\\
& $66 \to 79 \to 137 \to 66$ \\
& $92 \to 92$\\
\hline
%========
8
& $26 \to 48 \to 88 \to 136 \to 54 \to 49 \to 105 \to 34 \to 33 \to 26$ \\
\hline
%========
9& $10 \to 10$\\
& $11 \to 11$\\
& $34 \to 34$\\ 
& $46 \to 61 \to 46$\\
& $74 \to 74$\\
& $90 \to 90$\\
& $91 \to 91$\\
\hline
%========
% 10& $23 \to 23$\\
% & $48 \to 90 \to 91 \to 92 \to 95 \to 116 \to 48$\\
% & $83 \to 83$\\
% \hline
\end{tabular}
\caption{Fixed points and cycles of $S_{[c,10]}$ 
for $0 \leq c \leq 9$\label{T:10cycles}}
}
\end{center}
\end{table}

It follows from Lemma~\ref{L:bDescent} that, for any $c < 27$, for all $a \geq 100$, $S_{[c,10]}(a) <  a$.
We use this result to determine the fixed points and cycles of $S_{[c,10]}$ for $0 \leq c\leq 9$, presenting our results in Table~\ref{T:10cycles}.  

As the table illustrates, varying the constant greatly affects the behavior of $\Scb$ under iteration.  As one would expect, changing the base also changes the behavior, but, interestingly, there are patterns that occur when changing both the constant and the base.  For example, we show that, if $c$ and $b$ are both odd, then $\Scb$ has no fixed points.
First, we need a key lemma that we use repeatedly, throughout the paper.

\begin{lemma}\label{L:bodd}
If $b$ is odd, then $\Scb^k(a) \equiv kc + a \pmod 2$.
\end{lemma}

\begin{proof}
Let $b$ be odd and let $a = \sum_{i=0}^n a_ib^i$, as usual. Since $b$ is odd, $a  \equiv \sum_{i=0}^n a_i \pmod 2$, and therefore,
\[
\Scb(a) =  c + \sum_{i=0}^n a_i^2 \equiv c + \sum_{i=0}^n a_i  \equiv c + a \pmod 2.
\]
A simple induction argument completes the proof.
\end{proof}

\begin{theorem}\label{T:cbodd}
If $c$ and $b$ are both odd, then $\Scb$ has no fixed points and all of its cycles are of even length.
\end{theorem}

\begin{proof}
By Lemma~\ref{L:bodd}, $\Scb^k(a) \equiv kc + a \equiv k + a \pmod 2$, since $c$ is odd.  Thus, if $\Scb^k(a) = a$, then $k$ is even.  The result follows.
\end{proof}

Lemma~\ref{L:bDescent} also allows us to compute all fixed points and cycles for an  arbitrary constant $c$ and base $b$.  Of particular interest in Section~\ref{S:conj} is the case where both $c$ and $b$ are odd.  We provide lists of the fixed points and cycles of $\Scb$ in this case, for small values of $c$ and $b$ in Table~\ref{T:bcycles}.  We label some of the sequences for ease of reference in the proof of Theorem~\ref{T:conjsmall}.

Recall that a positive integer $a$ is a happy number if $S_{[0,10]}^k(a) = 1$ for some $k\in \ZZ^+$. We now generalize this idea to values of $c > 0$, noting that in these cases, $1$ is no longer a fixed point (nor in a cycle).

\begin{definition}\label{D:Ucb}
Fix $c\geq 0$ and $b \geq 2$.
Let $\Ucb$ denote the set of all fixed points and cycles of $\Scb$. That is, \[\Ucb = \{a \in \ZZ^+ \mid \Scb^m(a) = a \mbox{~for~some~} m\in\ZZ^+\}.\]
For $u \in \Ucb$, a positive integer $a$ is a {\it $u$-attracted number (for $\Scb$)} if $\Scb^k(a) = u$, for some $k\in \Znn$.
\end{definition}

\begin{table}[tbh]
\begin{center}
{\small
\begin{tabular}{|c|l|}\hline
$[c,b]$&Cycles of $\Scb$\\
\hline\hline
%========
$[1,3]$ & $ 12 \to 20 \to 12$\\
\hline
%========
$[3,3]$ & $ 22 \to 102 \to 22$\\
\hline
%========
$[5,3]$ & $20 \to 100 \to 20$\hfill $C_1$\\
& $ 21 \to 101 \to 21$\hfill $C_2$\\
& $22 \to 111 \to 22$\hfill $C_3$\\
\hline
%========
$[7,3]$ & $ 22 \to 120 \to 110 \to 100 \to 22$\\
\hline
%========
$[9,3]$ & $ 112 \to 120 \to 112$\\
\hline \hline
%========
$[1,5]$ & $2 \to 10 \to 2$ \\
\hline
%========
$[3,5]$ & $23 \to 31 \to 23$ \\
\hline
%========
$[5,5]$ & $11 \to 12 \to 20 \to 14 \to 42 \to 100 \to 11$ \\
\hline
%========
$[7,5]$ & $23 \to 40 \to 43 \to 112 \to 23$ \\
\hline
%========
$[9,5]$ & $21 \to 24 \to 104 \to 101 \to 21$ \\
\hline \hline
%========
$[1,7]$ & $13 \to 14 \to 24 \to 30 \to 13$\hfill $C_1$\\
& $35 \to 50 \to 35$\hfill $C_2$\\
\hline
%========
$[3,7]$ & $25 \to 44 \to 50 \to 40 \to 25$\hfill $C_1$\\
& $26 \to 61 \to 55 \to 104 \to 26$\hfill $C_2$\\
\hline
%========
$[5,7]$ & $6 \to 56 \to 123 \to 25 \to 46 \to 111 \to 11 \to 10 \to 6$\hfill $C_1$\\
&$13 \to 21 \to 13$\hfill $C_2$\\
&$34 \to 42 \to 34$\hfill $C_3$\\
\hline
%========
$[7,7]$ & $ 26 \to 65 \to 125 \to 52 \to 51 \to 45 \to 66 \to 142 \to 40 \to 32 \to 26$\\
\hline
%========
$[9,7]$ & $ 46 \to 115 \to 51 \to 50 \to 46$\\
\hline \hline
%========
$[1,9]$ & $3 \to 11 \to 3$ \\
\hline
%========
$[3,9]$ & $4 \to 21 \to 8 \to 74 \to 75 \to 85 \to 112 \to 10 \to 4$ \\
\hline
%========
$[5,9]$ & $25 \to 37 \to 70 \to 60 \to 45 \to 51 \to 34 \to 33 \to 25$\hfill $C_1$\\
& $28 \to 81 \to 77 \to 124 \to 28$\hfill $C_2$\\
& $46 \to 63 \to 55 \to 61 \to 46$\hfill $C_3$ \\
& $88 \to 157 \to 88$\hfill $C_4$ \\
\hline
%========
$[7,9]$ & $8 \to 78 \to 143 \to 36 \to 57 \to 100 \to 8$\hfill $C_1$\\
& $25 \to 40 \to 25$\hfill $C_2$\\
& $45 \to 53 \to 45$\hfill $C_3$\\
& $48 \to 106 \to 48$\hfill $C_4$\\
\hline
%========
$[9,9]$ & $12 \to 15 \to 38 \to 101 \to 12$\hfill $C_1$ \\
& $24 \to 32 \to 24$\hfill $C_2$\\
\hline
\end{tabular}
\caption{Cycles of $S_{[c,b]}$ 
for %$c$, $b\leq 9$ and odd.
$1 \leq c \leq 9$ odd and $2 \leq b \leq 9$ odd.
\label{T:bcycles}}}
\end{center}
\end{table}

For example, referring to Table~\ref{T:10cycles}, for $S_{[4,10]}$, we see that 40 is 6-attracted, as is 20, 8, and the other numbers in that cycle.  All of those numbers are also 40-attracted, etc.  And since $S_{[4,10]}(2) = 8$, 2 is also 6-attracted. Similarly, 42 is 24-attracted, since $S_{[4,10]}(42) = 24$.

%%%%%S:cbGoodness%%%%%
%%%%%S:cbGoodness%%%%%
\section{Consecutive $\boldsymbol{u}$-Attracted Numbers}\label{S:cbGoodness}
%%%%%S:cbGoodness%%%%%
%%%%%S:cbGoodness%%%%%

In this and the following section, we consider the existence of sequences of consecutive $u$-attracted numbers, for $u\in \Ucb$ for some fixed $c$ and $b$.  
As can be seen from Table~\ref{T:10cycles}, for $S_{[3,10]}$, 19 and 20 are consecutive 7-attracted numbers.  (Direct calculation shows that 1 and 2  are as well.)
% As noted in Table~\ref{T:10cycles}, 63 and 64 are consecutive $35$-attracted integers under $S_{[1,10]}$.  
Does such a consecutive pair exist for every choice of $c\geq 0$, $b\geq 2$, and $u\in \Ucb$?  Do there exist longer consecutive sequences of $u$-attracted numbers?
The analogous questions were answered for happy numbers in~\cite{ES} and for many cases of $e$-power $b$-happy numbers in~\cite{consec,consecsmall}. 

In light of Lemma~\ref{L:bodd}, when $c$ is even and $b$ is odd, for each $a$ and $k$, $\Scb^k(a) \equiv a\pmod 2$.  Thus, in these cases, there cannot exist a consecutive pair of (or longer sequences of consecutive) $u$-attracted numbers.  With this in mind, we introduce the following definition, found in~\cite{consec}.

\begin{definition}\label{D:2consecutive}
A sequence of positive integers is {\it$d$-consecutive} if it is an arithmetic sequence with constant difference $d$.
\end{definition}

By~\cite[Corollary 2]{consec}, given $b \geq 2$ and $d = \gcd(2,b - 1)$, there exist arbitrarily long finite $d$-consecutive sequences of 1-attracted numbers for $S_{[0,b]}$.  Adapting these ideas to augmented generalized happy functions, we prove the following theorem.

\begin{theorem}\label{T:happysequence}
Let $c\geq 0$, $b \geq 2$, and $u\in \Ucb$ be fixed.  Set $d = \gcd(2,b - 1)$.  Then there exist arbitrarily long finite sequences of $d$-consecutive $u$-attracted numbers for $\Scb$. 
\end{theorem}

As noted earlier, Lemma~\ref{L:bodd} shows that if $c$ is  even and $b$ is odd, then there does not exist any consecutive $u$-attracted numbers for $\Scb$.  But for $c$ and $b$ both odd, we conjecture that there are arbitrarily long finite sequences of consecutive $u$-attracted numbers for $\Scb$.

\begin{conjecture}\label{C:happysequencebetter}
Let $c > 0$ and $b > 2$ both be odd.  Then for each $u \in \Ucb$, there exist arbitrarily long finite sequences of consecutive $u$-attracted numbers for $\Scb$. 
\end{conjecture}

We now prove Theorem~\ref{T:happysequence}.
In Section~\ref{S:conj}, we prove special cases of Conjecture~\ref{C:happysequencebetter}, specifically the cases with both $c$ and $b$ less than 10.  

Our proof of Theorem~\ref{T:happysequence} follows the general outline of the proofs  in~\cite{consec}.  We note that, similar to the proofs in that work, our proofs lead to the somewhat stronger result in which $u\in \Ucb$ is replaced by $u$ in the image of $\Scb$. 
We begin with a definition and two important lemmas.

\begin{definition}\label{D:cbGood}
A finite set, $T\subset \ZZ^+$, is {\it $[c,b]$-good} if for each $u \in \Ucb$, there exists $n, k \in \Znn$ such that for all $t \in T$, $\Scb^k(t+n) = u$.
\end{definition}

\begin{lemma}\label{L:bLemma4}
Let $F: \ZZ^+ \rightarrow \ZZ^+$ be the composition of a finite sequence of the functions $\Scb$ and $I$, where $I: \ZZ^+ \rightarrow \ZZ^+$ is defined by $I(t) = t+1$.  If $F(T)$ is $[c,b]$-good, then $T$ is $[c,b]$-good.
\end{lemma}

\begin{proof}
It suffices to show that if $I(T)$ is $[c,b]$-good then $T$ is $[c,b]$-good, and if $\Scb(T)$ is $[c,b]$-good then $T$ is $[c,b]$-good. First suppose that $I(T)$ is $[c,b]$-good. Then for each $u\in\Ucb$, there exist $n^\prime$, $k\in\Znn$ such that $\Scb^{k}(t + 1 + n^\prime) = u$ for all $t\in T$. Letting $n = n^\prime + 1$, it follows that $T$ is $[c,b]$-good.

Now, suppose that $\Scb(T)$ is $[c,b]$-good. Then for each $u\in\Ucb$, there exist $n^\prime$, $k^\prime\in\Znn$ such that $\Scb^{k^\prime}(\Scb(t) + n^\prime) = u$ for all $t\in T$. Let $r\in \ZZ$ be the number of base-$b$ digits of the largest element of $T$, and set
\[
n = \underbrace{11\ldots 11}_{n^\prime}\underbrace{00\ldots 00}_{r},
\]
in base $b$. Letting $k = k^\prime + 1$, we have
\[
\Scb^k(t+n) = \Scb^{k^\prime}(\Scb(t+n)) = \Scb^{k^\prime}(\Scb(t)+n^\prime) = u.
\]
Hence, $T$ is $[c,b]$-good.
\end{proof}

\begin{lemma}\label{L:bLemma5}
If $T = \{t\}$, then $T$ is $[c,b]$-good.
\end{lemma}

\begin{proof}
Let $u\in\Ucb$. Then there exists some $v\in\ZZ^+$ such that $\Scb(v) = u$. Let $r\in\Znn$ be such that $t \leq b^rv$ and define $n=b^rv-t$ and $k = 1$. This yields $\Scb^k(t+n) = \Scb(v)=u$. Hence, $T$ is $[c,b]$-good.
\end{proof}

\begin{theorem}\label{T:bTheorem1}
Given $c \geq 0$ and $b\geq 2$, let $d = \gcd(2,b-1)$.  A finite set $T$ of positive integers is $[c,b]$-good if and only if all elements of $T$ are congruent modulo $d$.
\end{theorem}

\begin{proof}
Fix $c$, $b$, $d$, and nonempty $T$ as in the theorem.

First suppose that $T$ is $[c,b]$-good and let $t_1$, $t_2\in T$. If $b$ is even, then  $d = 1$ and so $t_1 \equiv t_2 \pmod d$, trivially.  If $b$ is odd, fix $u \in \Ucb$ and let $n$, $k \in \Znn$ be such that $\Scb^k(t_i + n) = u$ for $i = 1$, $2$.  So $\Scb^k(t_1 + n) = \Scb^k(t_2 + n)$. Applying Lemma~\ref{L:bodd}, we obtain 
$kc + t_1 + n \equiv kc + t_2 + n \pmod 2$, so that $t_1 \equiv t_2 \pmod 2$.  Thus, all of the elements of $T$ must be congruent modulo $2 = \gcd(2,b-1) = d$.

Conversely, assume that all of the elements of $T$ are congruent modulo $d$. Note that 
if $T$ has exactly one element, then, by Lemma~\ref{L:bLemma5}, $T$ is $[c,b]$-good. So we may assume that $|T| > 1$.
    
Letting $N = |T|$, assume by induction that any set of fewer than $N$ elements, all of which are congruent modulo $d$, is $[c,b]$-good.  Let $t_1$, $t_2 \in T$ be distinct, and assume without loss of generality that $t_1 > t_2$.  We will construct a function $F$, a finite composition of the functions $I$ and $\Scb$, so that $F(t_1) = F(t_2)$.

Consider the following three cases:
If it is the case that $t_1$ and $t_2$ have the same nonzero digits, we construct $F_1$ so that $F_1(t_1) = F_2(t_2)$. If $t_1 \equiv t_2 \pmod{b-1}$, then we construct $F_2$ so that $F_2(t_1)$ and $F_2(t_2)$ have the same nonzero digits. If it is neither the case that $t_1$ and $t_2$ have the same nonzero digits nor $t_1 \equiv t_2 \pmod{b-1}$, we construct $F_3$ so that $F_3(t_1) \equiv F_3(t_2)\pmod{b-1}$. Composing some or all of these functions will yield the desired function $F$.
    
{\it Case 1.} If $t_1$ and $t_2$ have the same nonzero digits, it follows from the definition of $\Scb$ that $\Scb(t_1) = \Scb(t_2)$.  In this case, let $F=F_1=\Scb$. 

{\it Case 2.} If $t_1 \equiv t_2 \pmod{b-1}$, then there is $v \in \ZZ^+$ so that $t_1-t_2 = (b-1)v$.  Let $r \in \ZZ^+$ be such that
$b^r > bv+t_2-v$ 
so that $b^r > bv$ and $b^r>t_2-v$. Let $m=b^r+v-t_2$ and note that $m>0$.  Then 
    \[I^m(t_1) = t_1 + b^r + v - t_2 = b^r + v +(b-1)v = b^r + bv\]
    and \[I^m(t_2) = t_2 + b^r + v -t_2 = b^r + v.\]
Since $r$ was chosen so that $b^r > bv$, $I^m(t_1)$ and $I^m(t_2)$ must have the same nonzero digits as in Case 1. In this case, let $F_1=\Scb$ and $F_2=I^m$ and let $F=F_1 \circ F_2$.

{\it Case 3.} If neither of the above hold, let $w = t_1-t_2$. We first show that there exists $0 \leq j < b - 1$ such that 
\begin{equation}\label{E:congruence}
2j \equiv -\Scb(w - 1) + c - 1 \pmod{b - 1}
\end{equation}
(where, for convenience, we define $\Scb (0) = 0$).
If $b$ is even, then $2$ and $b-1$ are relatively prime and so $2$ is invertible modulo $b-1$ and such a $j$ clearly exists.  If $b$ is odd, then $d = 2$ and so $w$ is even.  By Lemma~\ref{L:bodd}, 
$\Scb(w - 1) \equiv c + w - 1 \equiv c - 1 \pmod 2$ and so 
$-\Scb(w - 1) + c - 1$ is even.  Thus, there exists $j^\prime\in \ZZ$ such that
$2j^\prime = -\Scb(w - 1) + c - 1$.  Hence there is a $0 \leq j < b$ satisfying~\eqref{E:congruence}.

Now choose $r^{\prime} \in \ZZ^+$ such that $(j+1)b^{r^{\prime}} > t_1$ and let $m^{\prime} = (j+1)b^{r^{\prime}} - t_2 - 1$. Note that $m^\prime \geq 0$. Then, 
    \begin{align*}
    \Scb(t_1+m^{\prime}) &= \Scb((j+1)b^{r^{\prime}} + w -1)\\
    &= (j+1)^2 + \Scb(w-1) \\
    &= (j^2+2j+1)+\Scb(w-1) \\
    &\equiv j^2 + c \pmod{b-1}
    \end{align*}
and
    \begin{align*}
    \Scb(t_2+m^{\prime}) &= \Scb((j+1)b^{r^{\prime}}-1) \\
    &= j^2+(b-1)^2r^{\prime}+c\\
    &\equiv j^2 +c \pmod{b-1}.
    \end{align*}
Therefore $\Scb I^{m^{\prime}}(t_1) \equiv \Scb I^{m^{\prime}}(t_2) \pmod{b-1}$ as in Case 2. In this case, let $F_1=\Scb$, $F_2=I^m$, and $F_3=\Scb I^{m^{\prime}}$ (with $m$ chosen as in Case 2). Further, let $F=F_1 \circ F_2 \circ F_3$.

Therefore, there exists a function $F$, a composition of a finite sequence of $\Scb$ and $I$, such that $F(t_1) = F(t_2)$ and so $|F(T)|<|T|$. By the induction hypothesis, it follows that $F(T)$ is $[c,b]$-good.  Finally, by Lemma~\ref{L:bLemma4}, $T$ is $[c,b]$-good.
\end{proof}

We are now ready to prove the main theorem of this paper.

\begin{proof}[Proof of Theorem~\ref{T:happysequence}]
Let $c$, $b$, and $u$ be given, and let $N \in \ZZ^+$ be arbitrary.  

If $b$ is even, set 
$T = \{t\in \ZZ \mid 1 \leq t \leq N\}$.  By Theorem~\ref{T:bTheorem1}, since $d = 1$, $T$ is 
$[c,b]$-good.  Thus there exist $n$, $k\in \Znn$ such that, for each $t\in T$, $\Scb^k(t + n) = u$.  Hence the set 
$\{t + n\in \ZZ \mid 1 \leq t \leq N\}$ is a sequence of $N$ consecutive $u$-attracted numbers.

If $b$ is odd, then let $T = \{2t\in \ZZ \mid 1 \leq t \leq N\}$.  Then by Theorem~\ref{T:bTheorem1}, since $d = 2$, $T$ is $[c,b]$-good and, as above, the set 
$\{2t + n\in \ZZ \mid 1 \leq t \leq N\}$ is a sequence of $N$ $2$-consecutive $u$-attracted numbers.
\end{proof}

\section{Special Cases of Conjecture~\ref{C:happysequencebetter}}
\label{S:conj}

By Theorem~\ref{T:happysequence} if $b$ is odd, then there are arbitrarily long finite $2$-consecutive sequences of $u$-attracted numbers. By Lemma~\ref{L:bodd}, if, in addition, $c$ is even, then there cannot exist any nontrivial consecutive sequences of $u$-attracted numbers. This leaves the existence of such sequences undetermined in the case of both $c$ and $b$ odd.

In this section, we prove that Conjecture~\ref{C:happysequencebetter} holds for values of $c$ and $b$ both less than 10.

\begin{theorem}\label{T:conjsmall}
Let $1\leq c \leq 9$ and $3\leq b\leq 9$   be odd and let $u\in U_{[c,b]}$.  Then there exist arbitrarily long finite sequences of consecutive $u$-attracted numbers for $S_{[c,b]}$. 
\end{theorem}

By Theorem~\ref{T:bTheorem1}, no set containing even two consecutive integers can be
$[c,b]$-good.  Hence, to prove Theorem~\ref{T:conjsmall}, we need a new, similar, property, which we define below.

\begin{definition}\label{D:cbCycleGood}
A finite set, $T$, is {\it $[c,b]$-cycle-good} if, for each cycle $C$ of $\Scb$, there exists $n, k \in \ZZ^+$ such that for all $t \in T$, $\Scb^k(t+n) \in C$.
\end{definition}

Note that any $[c,b]$-good set is necessarily $[c,b]$-cycle-good, but not the converse.
We need the following analog of Lemma~\ref{L:bLemma4}, with  $[c,b]$-cycle-good in place of $[c,b]$-good.  Its proof completely parallels that of Lemma~\ref{L:bLemma4} and so is omitted. 

\begin{lemma}\label{L:bLemma4cycle}
Let $F: \ZZ^+ \rightarrow \ZZ^+$ be the composition of a finite sequence of the functions, $\Scb$ and $I$, where $I: \ZZ^+ \rightarrow \ZZ^+$ is defined by $I(t) = t+1$.  If $F(T)$ is $[c,b]$-cycle-good, then $T$ is $[c,b]$-cycle-good.
\end{lemma}

\begin{proof}[Proof of Theorem~\ref{T:conjsmall}]
First note that if $\Scb$ has only one cycle, then for each $u\in \Ucb$, every positive integer is $u$-attracted.  Hence, the theorem holds in these cases.  Thus, using Table~\ref{T:bcycles}, it remains to prove the conjecture for each $[c,b]$ in the set
\[A = \left\{[5,3], [1,7], [3,7], [5,7], [5,9], [7,9],[9,9]\right\}.
\]

Fix $[c,b]\in A$ and let $T$ be a nonempty finite set of positive integers.  We now prove that $T$ is $[c,b]$-cycle-good.  

To set notation, let $T_e$ be the set of all even elements of $T$ and let $T_o$ be the set of all odd elements of $T$.  We assume that neither $T_e$ nor $T_o$ is empty, since, otherwise, by Theorem~\ref{T:bTheorem1}, $T$ is $[c,b]$-good and thus $[c,b]$-cycle-good.  Let the constant $v$ and the sets $V_j$ be as given in Table~\ref{T:conjpf2}, and let $\ell$ be the length of $C_1$, the cycle of $\Scb$ containing $v$ (as seen in Table~\ref{T:bcycles}).

\begin{table}[H]
\begin{center}
{\small
\begin{tabular}{|c|c|l|}\hline
$[c,b]$&$v$&$V_j$\\
\hline\hline
%========
$[5,3]$& 20 & $V_1 = \{20,100\}$, $V_2 = \{20,21\}$, $V_3 = \{20,111\}$\\
\hline
$[1,7]$& 13 & $V_1 = \{13,14\}$, $V_2 = \{13,30\}$, $V_3 = \{13,50\}$ \\
\hline
$[3,7]$& 44 & $V_1 = \{44,25\}$, $V_2 = \{44,50\}$, $V_3 = \{44,61\}$, \\
& & $V_4 = \{44,104\}$  \\
\hline
$[5,7]$& 6 & $V_1 = \{6,10\}$,  $V_2 = \{6, 21\}$, $V_3 = \{6, 25\}$, \\
& & $V_4 = \{6, 34\}$, $V_5 = \{6, 56\}$, $V_6 = \{6, 111\}$\\
\hline
$[5,9]$& 37 & $V_1=\{37, 25\}$, $V_2 = \{37, 34\}$, $V_3=\{37, 45\}$, \\
& & $V_4=\{37, 61\}$, $V_5=\{37, 63\}$, $V_6=\{37, 70\}$ \\
& & $V_7=\{37, 81\}$, $V_8=\{37, 124\}$, $V_9=\{37, 157\}$\\
\hline
$[7,9]$& 8 &$V_1 = \{8,25\}$, $V_2 = \{8,36\}$, $V_3 = \{8,45\}$, \\
& & $V_4 = \{8,78\}$, $V_5 = \{8,100\}$, $V_6 = \{8,106\}$\\
\hline
$[9,9]$& 15 & $V_1 =\{15, 12\}$, $V_2=\{15, 32\}$, $V_3=\{15, 38\}$ \\
\hline
\end{tabular}
\caption{Values for the proof of Theorem~\ref{T:conjsmall}
\label{T:conjpf2}}}
\end{center}
\end{table}

By Theorem~\ref{T:bTheorem1}, the set $T_e$ is $[c,b]$-good.  Thus, there exist positive integers $n_1$ and $k_1$ such that for each $t\in T_e$, $S_{[c,b]}^{k_1}(t + n_1) = v$.  Let $T^\prime = \{S_{[c,b]}^{k_1}(t + n_1) \mid t \in T_o\}$.  It follows from Lemma~\ref{L:bodd} that the elements of $T^\prime$ are all congruent modulo 2.  Hence, by Theorem~\ref{T:bTheorem1}, the set $T^\prime$ is also $[c,b]$-good.  Thus, there exist positive integers $n_2$ and $k_2$ such that for each $t\in T^\prime$, $S_{[c,b]}^{k_2}(t + n_2) = v$.  Combining these results, we find that
\[S_{[c,b]}^{k_2}\left(S_{[c,b]}^{k_1}(T + n_1) + n_2\right) = \left\{v,S_{[c,b]}^{k_2}(v + n_2)\right\},\]
where, again using Lemma~\ref{L:bodd}, since $v$ is even, $S_{[c,b]}^{k_2}(v + n_2)$ is odd.  Let $k_3$ be a multiple of $\ell$, sufficiently large so that $S_{[c,b]}^{k_2 + k_3}(v + n_2) \in U_{[c,b]}$.  Then 
\[S_{[c,b]}^{k_2 + k_3}\left(S_{[c,b]}^{k_1}(T + n_1) + n_2\right) = \left\{v,S_{[c,b]}^{k_2 + k_3}(v + n_2)\right\} = V_j,\]
for some $j$.

\begin{table}[tbh]
\begin{center}
{\small
\begin{tabular}{|c|c|c|c|c|}\hline
$[c,b]$&$k_1,n_1$&$k_2,n_2$&$k_3,n_3$&$k_4,n_4$\\
\hline\hline
%========
$[5,3]$&$0,0$&$0,1$&$2,11120200$ & -\\
\hline
$[1,7]$&$0,0$&$1,1111111$ & - & -\\
\hline
$[3,7]$&$0,0$&$0,30$ & - & - \\
\hline
$[5,7]$&$0,0$&$1,3$&$1,121303$& - \\
\hline
$[5,9]$&$0,0$&$2,81$&$5,11$&$1,156155$ \\
\hline
$[7,9]$&$3,2$&$6,131$&$4,135$&$3,13$\\
\hline
$[9,9]$&$0,0$&$3,218$ & - & -\\
\hline
\end{tabular}
\caption{$\Scb^{k_i}(V_1 + n_i) \subseteq C_i$
\label{T:conjpf3}}}
\end{center}
\end{table}

{\small \begin{table}[H]
\begin{center}
\begin{tabular}{|c|c|c|c|c|c|c|c|c|}\hline
$[c,b]$&$k_2^\prime,n_2^\prime$&$k_3^\prime,n_3^\prime$&$k_4^\prime,n_4^\prime$&$k_5^\prime,n_5^\prime$&$k_6^\prime,n_6^\prime$&$k_7^\prime,n_7^\prime$&$k_8^\prime,n_8^\prime$&$k_9^\prime,n_9^\prime$\\
\hline\hline
%========
$[5,3]$&$2,1202$&$2,12112$&&&&&& \\
\hline
$[1,7]$&$1,0$&$5,14$&&&&&& \\
\hline
$[3,7]$&$3,0$ &$5,131$&$3,3$&&&&&\\
\hline
$[5,7]$&$9,1$&$5,16$&$6,114$&$7,0$&$6,114$&&& \\
\hline
$[5,9]$&$9,212$&$4,7$&$6,22$&$5,147$&$7,0$& $1,212$&$6,147$&$7,32$\\
\hline
$[7,9]$&$9,3$&$6,150$&$4,16$&$11,31$&$9,31$&&&\\
\hline
$[9,9]$&$7,6$&$3,0$&&&&&&  \\
\hline
\end{tabular}
\caption{$\Scb^{k_j^\prime}(V_j + n_j^\prime) \subseteq V_1$
\label{T:conjpf4}}
\end{center}
\end{table}}

By Lemma~\ref{L:bLemma4cycle}, to prove that $T$ is $[c,b]$-cycle-good, it suffices to prove that each of the $V_j$ is $[c,b]$-cycle-good.  To each $C_i$, we associate a pair $(k_i,n_i)$, as given in Table~\ref{T:conjpf3} and to each $V_j$, we associate a pair $(k_j^\prime,n_j^\prime)$, as in Table~\ref{T:conjpf4}.

To see that $V_1$ is [c,b]-cycle-good, fix a cycle, $C_i$, of $\Scb$ and note that for each $t \in V_1$, $\Scb^{k_i}(t + n_i) \in C_i$. Thus $V_1$ is $[c,b]$-cycle-good.  Now, fix $j$ such that $V_j \neq V_1$.  A direct calculation shows that for each $t \in V_j$, $\Scb^{k_j^\prime}(t + n_j^\prime) \in V_1$.  Since $V_1$ is $[c,b]$-cycle-good, so is $V_j$.  Thus, by Lemma~\ref{L:bLemma4cycle}, $T$ is $[c,b]$-cycle-good.

Hence, for each $c$ and $b$ as in the theorem, every finite set is $[c,b]$-cycle-good.  Considering the specific sets $T_N = \{1,2,\dots,N\}$ completes the proof.
\end{proof}

\end{document}